\documentclass{amsart}

\usepackage{mathrsfs}

\usepackage{bbm}

\newtheorem{Thm}{Theorem}[section]
\newtheorem{Lem}[Thm]{Lemma}

\theoremstyle{definition}

\theoremstyle{remark}

\begin{document}

\title[the sum of a prime and a square-full number in short intervals]{Asymptotic formula for the sum of a prime and a square-full number in short intervals}
\author[F.Ogihara and Y.Suzuki]{Fumi Ogihara and Yuta Suzuki}
\maketitle\thispagestyle{empty}




\date{}

\dedicatory{}

\begin{abstract}
	Let $R_{m, \, \textup{sq-full}}(N)$ be a representation function for the sum of a prime and a square-full number.  
	In this article, we prove an asymptotic formula for the sum of $R_{m, \, \textup{sq-full}}(N)$ over positive integers $N$ in a short interval ($X$, $X+H$] of length $H$ slightly bigger than $X^{\frac{1}{2}}$.
	
\end{abstract}

\section{Introduction}
	As one of the well-known Hardy--Littlewood conjectures, it is conjectured that every sufficiently large positive number $N$ is a square or it can be expressed as the sum of a prime and a square. In other words, 
	$$
		N = n^{2} \quad \mbox{or} \quad p + n^{2}, 
	$$
	where $p$ denotes a prime number and $n$ denotes a positive integer.
	
	To this conjecture, an effective approach to count $N$ expressible as $p+n^{2}$ is to use the following function
	\begin{equation}\label{R_(p, sq)(N)}
		R_{p, \, \textup{sq}}(N) = \sum_{p+n^{2}=N} \log p
	\end{equation}
	or
	\begin{equation} \label{R_(m, sq)(N)}
		R_{m, \, \textup{sq}}(N) = \sum_{m+n^{2}=N} \Lambda(m), 
	\end{equation}
	where $\Lambda(m)$ is the von Mangoldt function given by
	\begin{equation*}
		\Lambda(m) = 
			\begin{cases}
				\log p & \mbox{if } m=p^{k} \mbox{ for some prime $p$ and some integer } k \ge 1, \\
				0 & \mbox{otherwise}. 
			\end{cases}
	\end{equation*}
	We call the functions such as ($\ref{R_(p, sq)(N)}$) and ($\ref{R_(m, sq)(N)}$) the representation functions.

	In \cite{H-L}, Hardy and Littlewood obtained a hypothetical asymptotic formula for $N$ not being a square,
	\begin{equation} \label{H-L_AP}
		R_{m, \, \textup{sq}}(N) = \mathfrak{S} \sqrt{N} + \mbox{(error)},
	\end{equation}
	where
	$$
		\mathfrak{S} = \mathfrak{S} (N) = \prod_{p > 2} \bigg( 1 - \frac{(N/p)}{p-1} \bigg)
	$$
	with ($N/p$) being the Legendre symbol. 
	It is difficult to obtain ($\ref{H-L_AP}$) itself, but it may be possible to consider the short interval average of the representation functions: 
	\begin{equation*} 
		\sum_{X < N \le X+H} R_{m, \, \textup{sq}}(N)
	\end{equation*}
	for $X$, $H$ being real numbers with $ 4 \le H \le X$. 
	
	In 2017, A. Languasco and A. Zaccagnini showed that
	
	\begin{Thm} {\rm (Languasco--Zaccagnini\cite{L-Z})}.
		For real numbers $X$, $H$ with $4 \le H \le X$ and any real number $\varepsilon >0$, there exists a constant $C>0$ such that
		\begin{equation} \label{L-Zresult}
			\sum_{X < N \le X+H} R_{m, \, \textup{sq}}(N) = HX^{\frac{1}{2}}+O \bigg( HX^{\frac{1}{2}} \exp \bigg( -C \bigg(\frac{\log X}{\log \log X} \bigg)^{\frac{1}{3}} \bigg) \bigg)
		\end{equation}
		provided 
		\begin{equation*} 
			X^{\frac{1}{2}} \exp \Big( -C \Big(\frac{\log X}{\log \log X} \Big)^{\frac{1}{3}} \Big) \le H \le X^{1- \varepsilon}.
		\end{equation*}
	\end{Thm}

	In 2023, Y. Suzuki improved the above theorem: 
	
	\begin{Thm} {\rm (Suzuki\cite{Suzuki})}.
		For $X^{\frac{32-4 \sqrt{15}}{49} + \varepsilon } \le H \le X^{1- \varepsilon}$, the asymptotic formula (\ref{L-Zresult}) holds. 
	\end{Thm}
	
	In this article, we study a variant of the above problem, where we replace squares with square-full numbers. 
	Recall that a positive integer $n$ is square-full if a prime $p$ divides $n$, then $p^{2}$ also divides $n$. 
	
	Let us mention some known results about square-full numbers.
	Let $Q(x)$ be a counting function of square-full numbers less than or equal to $x$.
	Then we have following the two theorems.
	\begin{Thm} {\rm (Bateman--Grosswald\cite{B-G})}. \label{Thm_B-G}
		For a real number $x \ge 1$, we have
		$$
			Q(x) = \frac{\zeta(\frac{3}{2})}{\zeta(3)}x^{\frac{1}{2}} + \frac{\zeta(\frac{2}{3})}{\zeta(3)}x^{\frac{1}{3}} + O(x^{\frac{1}{6}}) .
		$$
	\end{Thm}
	\begin{Thm} {\rm (Filaseta--Trifonov\cite{F-T})}. \label{Thm_F-T}
		Under $x^{\frac{1}{2}+ \frac{5}{39} +\varepsilon} \le H \le x$, we have
		\begin{equation*} 
			Q(x+H) - Q(x) \sim \frac{\zeta(\frac{3}{2})}{2\zeta(3)}Hx^{- \frac{1}{2}}.
		\end{equation*}
	\end{Thm}

	In 2002, O. Trifonov improved the exponent of $x$ in Theorem $\ref{Thm_F-T}$ from $\frac{5}{39}$ to $\frac{19}{154}$ (\cite{T}).
	
	It follows easily that
	\begin{equation} \label{sq_are_sqfull}
		R_{p, \, \textup{sq}}(N) = \sum_{m+n^{2}=N} \Lambda(m) \le \sum_{\substack{m+f=N \\ f : \textup{square-full}}}\Lambda(m)
	\end{equation}
	since squares are square-full numbers.
	In this article, we consider the evaluation of the most-right-hand side of ($\ref{sq_are_sqfull}$): 
	$$
		R_{m, \, \textup{sq-full}}(N) := \sum_{\substack{m+f=N \\ f \in \mathscr{Q} }} \Lambda(m),
	$$
	where $\mathscr{Q} := \{ n \in \mathbb{N} \, | \, n \mbox{ is square-full} \} $.
	Especially, we consider the short interval average of the $R_{m, \, \textup{sq-full}}(N)$.
	An easy calculation shows
	\begin{equation} \label{caluR_(m, sq-full)(N)}
		\sum_{X < N \le X+H} R_{m, \, \textup{sq-full}}(N)
			= \hspace{-1em} \sum_{\substack{X < m+f \le X+H \\ f \in \mathscr{Q}}} \hspace{-1em} \Lambda(m) 
			= \sum_{m \le X+H} \Lambda(m) \hspace{-1em} \sum_{\substack{X-m < f \le X-m+H \\ f \in \mathscr{Q}}} 1 .
	\end{equation}
	
	By using Theorem \ref{Thm_F-T}, for $x^{\frac{1}{2}+ \theta +\varepsilon} \le H \le x$ ($ \theta = \frac{5}{39}, \, \frac{19}{154}$), (\ref{caluR_(m, sq-full)(N)}) roughly becomes
	\begin{equation} \label{useF-T}
			\begin{split}
				\sum_{X < N \le X+H} R_{m, \, \textup{sq-full}}(N) &\sim \sum_{m \le X+H} \Lambda(m) \bigg( \frac{\zeta(\frac{3}{2})}{2\zeta(3)}HX^{- \frac{1}{2}} \bigg) \\
									  &\sim \frac{\zeta(\frac{3}{2})}{\zeta(3)}HX^{\frac{1}{2}} \quad (X \to \infty).
			\end{split}
	\end{equation}
	
	Our main theorem is the following, that is we improve the range of $H$ in (\ref{useF-T}) to $\theta = 0$:
	\begin{Thm} \label{main}
		For $X^{\frac{1}{2}+\varepsilon} \le H \le X^{1- \varepsilon}$, we have
		$$
			\sum_{X < N \le X+H} R_{m, \;\textup{sq-full}}(N) = \frac{\zeta(\frac{3}{2})}{\zeta(3)}HX^{\frac{1}{2}} \bigg( 1 + O \Big( (\log X)^{-1} \Big) \bigg) .
		$$
	\end{Thm}

\section{Proof of the main theorem}
	Let $\mathbbm{1}_{P}$ be a characteristic function for some condition $P$.
	\begin{Lem} \label{Lem_chara}
		For a positive number f, we have
		$$
			\mathbbm{1}_{f \in \mathscr{Q}} = \sum_{a^{2}b^{3} = f} \mu(b)^{2},
		$$
		where $\mu(b)$ is the M\"{o}bius function.
	\end{Lem}
	\begin{proof}
		This is well-known. 
		It suffices to prove that any square-full number $f$ can be written uniquely as $f = a^{2}b^{3}$ with a positive integer $a$ and a square-free number $b$.
		For positive numbers $n$ and $d$, we write
		$$
			d \, || \, n 
			\quad \mbox{if} \quad 
			d \, | n \hspace{0.5em} \mbox{and} \hspace{0.5em} \Big( d, \, \frac{n}{d} \, \Big) = 1 .
		$$
		Then, we have
		$$
				f = \prod_{p^{v} || f} p^{v} 
				  = \prod_{\substack{p^{v} || f \\ v : \textup{even}}} p^{v} \prod_{\substack{p^{v} || f \\ v : \textup{odd}}} p^{v}.
		$$
		If $v$ is odd in $p^{v} || f$, we have $v \ge 3$ because $f$ is square-full. This gives,
		$$
			f = \prod_{\substack{p^{v} || f \\ v : \textup{even}}} p^{v} \prod_{\substack{p^{v} || f \\ v : \textup{odd}}} p^{v-3} \bigg( \prod_{\substack{p^{v} || f \\ v : \textup{odd}}} p \bigg)^{3}.
		$$
		Then 
		$\displaystyle \prod_{\substack{p^{v} || f \\ v : \textup{even}}} p^{v} \prod_{\substack{p^{v} || f \\ v : \textup{odd}}} p^{v-3}$ 
		is square because $v-3$ is even when $v$ is odd and
		$\displaystyle \prod_{\substack{p^{v} || f \\ v : \textup{odd}}} p$ 
		is square-free because the exponent of each prime factor $p$ is $1$.	
	\end{proof}
	
	For a real number $B \ge 2$, we define
	\begin{equation} \label{R_{B}}
			R_{B}(N) : = \sum_{\substack{m + a^{2}b^{3} = N \\ b \le B}} \Lambda(m) \mu(b)^{2}.
	\end{equation}
	
	\begin{Lem} \label{Lem_R(N)=R_{B}(N)+(error)}
		For a positive number $N$ and a real number $B \ge 1$, we have
		$$
			R_{m, \, \textup{sq-full}}(N) = R_{B}(N) + O( (N^{\frac{1}{2}} \log N ) B^{- \frac{1}{2}} ).
		$$
	\end{Lem}
	\begin{proof}
		Using Lemma \ref{Lem_chara}, we have 
		\begin{equation*}
			\begin{split}
				R_{m, \, \textup{sq-full}}(N) & = \sum_{\substack{m+f=N \\ f \in \mathscr{Q} }} \Lambda(m) 
						   = \sum_{m + a^{2}b^{3} = N} \Lambda(m) \mu(b)^{2}.
			\end{split}
		\end{equation*}
		The difference between $R_{m, \, \textup{sq-full}}(N)$ and $R_{B}(N)$ can be estimated as
		\begin{align*} 
			R_{m, \, \textup{sq-full}}(N) - R_{B}(N) &= \sum_{\substack{m+a^{2}b^{3}=N \\ b > B}} \Lambda(m) \mu(b)^{2} \\
							&\le \log N \sum_{\substack{m+a^{2}b^{3}=N \\ b > B}} 1 \\
							&= \log N \sum_{\substack{a^{2}b^{3} \le N \\ b > B}} 1 \\
							&= \log N \sum_{B < b \le N^{\frac{1}{3}}} \sum_{a^{2} \le \frac{N}{b^{3}}} 1 \\
							&\le N^{\frac{1}{2}} \log N \sum_{B < b} b^{- \frac{3}{2}} \\ 
							&\ll (N^{\frac{1}{2}} \log N) B^{-\frac{1}{2}}.
		\end{align*}
		This completes the proof.
	\end{proof}
	
	By Lemma \ref{Lem_R(N)=R_{B}(N)+(error)}, we have
	$$
		\sum_{X < N \le X+H} R_{m, \, \textup{sq-full}}(N) = \sum_{X \le N \le X+H} R_{B}(N) + O \Bigg( \sum_{X < N \le X+H} (N^{\frac{1}{2}} \log N ) B^{- \frac{1}{2}} \Bigg).
	$$
	If $B = (\log X)^{4}$, we have
	$$
		\sum_{X < N \le X+H} R_{m, \, \textup{sq-full}}(N) = \sum_{X \le N \le X+H} R_{B}(N) + O( HX^{\frac{1}{2}}(\log X)^{-1}).
	$$
	The following lemma is well-known.
	\begin{Lem} \label{Lem_cheby}
		For a real number $x$ such that $x \ge 2$, we have
		$$
			\sum_{m \le x} \Lambda(m) = \sum_{p \le x} \log p + O(x^{\frac{1}{2}} (\log x)^{2}).
		$$
	\end{Lem}
	
	Let $\mathscr{Q}_{B} := \{ n \in \mathbb{N} \, | \, n=a^{2}b^{3}, b \le B \mbox{ and } b \mbox{ is square-free} \}$. Then from (\ref{R_{B}})
	$$
		\sum_{X < N \le X+H} R_{B}(N) 
		= \sum_{\substack{X < m+a^{2}b^{3} \le X+H \\ b \le B}} \Lambda(m) \mu(b)^{2} 
		= \sum_{\substack{X < m+f \le X+H \\ f \in \mathscr{Q}_{B}}} \Lambda(m) .
	$$

	Using Lemma \ref{Lem_cheby}, we have
	\begin{align*}
		\sum_{X < N \le X+H} R_{B}(N) 
		=& \sum_{\substack{f \le X+H \\ f \in \mathscr{Q}_{B}}} \sum_{X-f < p \le X+H-f} \hspace{-1em} \log p 
			+ O \bigg( \sum_{\substack{f \le X+H \\ f \in \mathscr{Q}_{B}}} X^{\frac{1}{2}} (\log X)^{2} \bigg) \\
		=& \sum_{\substack{f \le X+H \\ f \in \mathscr{Q}_{B}}} \sum_{X-f < p \le X+H-f} \log p 
			\, + \, O(X(\log X)^{2}) \\
		=& \sum_{\substack{f \le X-2H \\ f \in \mathscr{Q}_{B}}} \sum_{X-f < p \le X+H-f} \log p \\
		 & \hspace{1em} + \hspace{-1em} \, \sum_{\substack{X-2H < f \le X+H \\ f \in \mathscr{Q}_{B}}} \sum_{X-f < p \le X+H-f} \hspace{-1em} \log p
			\, + \, O(X(\log X)^{2}).  \\
	\end{align*}
	For the second sum on the most-right-hand side can be bounded as
	\begin{equation*}
		\begin{split}
			\sum_{\substack{X-2H < f \le X+H \\ f \in \mathscr{Q}_{B}}} \sum_{X-f < p \le X+H-f} \log p 
			\ll & \,\, H \sum_{\substack{X-2H < f \le X+H \\ f \in \mathscr{Q}_{B}}} 1 \\
			\ll & \,\, H \sum_{b \le B} \sum_{X-2H < a^{2}b^{3} \le X+H} 1 \\
			\ll & \,\, H \sum_{b \le B} \bigg( \frac{H}{b^{3}} \bigg)^{\frac{1}{2}}
			\ll \,\, H^{\frac{3}{2}}. 
		\end{split}
	\end{equation*}
	Therefore, we get
	\begin{align}
		  & \sum_{X < N \le X+H} R_{B}(N) \notag \\
		=& \hspace{-0.5em} \sum_{\substack{f \le X-2H \\ f \in \mathscr{Q}_{B}}} \sum_{X-f < p \le X+H-f} \log p 
			\, + \, O(H^{\frac{3}{2}}B + X(\log X)^{2}) \notag \\
		=& \hspace{-0.5em} \sum_{m^{2} \le X-2H} \sum_{\substack{m^{2} \le f < (m+1)^{2} \\ f \le X-2H \\ f \in \mathscr{Q}_{B}}} \sum_{X-f < p \le X+H-f} \hspace{-2em} \log p 
			\, + \, O(H^{\frac{3}{2}} + X(\log X)^{2}).  \label{assign_cheby}
	\end{align}
	
	\begin{Lem} \label{Lem_int}
		For a positive integer $m$ and a square-full number $f$ with $m^{2} \le f < (m+1)^{2}$, we have
		$$
			\sum_{X-f < p \le X+H-f} \log p = \int_{m}^{m+1} \sum_{X-u^{2} < p \le X+H-u^{2}} \log p \, du + O(m\log X) .
		$$
	\end{Lem}
	\begin{proof}
		When $u \in [m, f^{\frac{1}{2}} ]$, we have 
		$$
			X-f < X-u^{2} \quad \mbox{and} \quad X+H-f < X+H-u^{2}.
		$$ 
		By recalling $f < (m+1)^{2}$, we have
		\begin{equation*}
			\begin{split}
				    & \bigg| \sum_{X-f < p \le X+H-f} \log p \quad - \sum_{X-u^{2} < p \le X+H-u^{2}} \log p \bigg| \\
				\le & \sum_{X-f < p \le X-u^{2}} \log p \quad+ \sum_{X+H-f < p \le X+H-u^{2}} \log p \\
				\ll & (f - u^{2} + 1) \log X \\
				\ll & ( (m+1)^{2} -m^{2} ) \log X \\
				\ll & m\log X.
			\end{split}
		\end{equation*}
		Similarly, if $u \in [f^{\frac{1}{2}}, m+1]$, we have 
		$$
			X-u^{2} < X-f \quad \mbox{and} \quad X+H-u^{2} < X+H-f
		$$
		and so
		\begin{equation*}
			\begin{split}
				    & \bigg| \sum_{X-f < p \le X+H-f} \log p \quad - \sum_{X-u^{2} < p \le X+H-u^{2}} \log p \bigg| \\
				\le & \sum_{X-u^{2} < p \le X-f} \log p \quad+ \sum_{X+H-u^{2} < p \le X+H-f} \log p \\
				\ll & (u^{2} - f+1) \log X \\
				\ll & ( (m+1)^{2} -m^{2} ) \log X \\
				\ll & m\log X.
			\end{split}
		\end{equation*}
		In any case, we thus have
		$$
			\sum_{X-f < p \le X+H-f} \log p = \sum_{X-u^{2} < p \le X+H-u^{2}} \log p + O(m\log X)
		$$
		for $m \le u \le m+1$. 
		By integrating this formula over $m \le u \le m+1$, we give the lemma.
		
	\end{proof}
	
	Substituting Lemma \ref{Lem_int} into (\ref{assign_cheby}) then, we have
	\begin{align*}
		 & \sum_{X < N \le X+H} \hspace{-0.5em} R_{B}(N) \\
		=& \sum_{m^{2} \le X-2H} \sum_{\substack{m^{2} \le f < (m+1)^{2} \\ f \le X-2H \\ f \in \mathscr{Q}_{B}}} \bigg( \int_{m}^{m+1} \sum_{X-u^{2} < p \le X+H-u^{2}} \log p \, du + O(m\log X) \bigg) \\
		 & \hspace{20em} + O( H^{\frac{3}{2}} + X(\log X)^{2}). \\
	\end{align*}
	By estimating the contribution of $O(m\log X)$ as
	$$
		\sum_{m^{2} \le X-2H} \sum_{\substack{m^{2} \le f < (m+1)^{2} \\ f \le X-2H \\ f \in \mathscr{Q}_{B}}} m\log X
		\ll X^{\frac{1}{2}}\log X \sum_{\substack{f \le X-2H \\ f \in \mathscr{Q}_{B}}} 1
		\ll X\log X.
	$$
	We obtain
	\begin{align}
		 & \sum_{X < N \le X+H} \hspace{-0.5em} R_{B}(N) \notag \\
		=& \sum_{m^{2} \le X-2H} \sum_{\substack{m^{2} \le f < (m+1)^{2} \\ f \le X-2H \\ f \in \mathscr{Q}_{B}}} \int_{m}^{m+1} \bigg( \sum_{X-u^{2} < p \le X+H-u^{2}} \hspace{-1.5em} \log p - H \bigg) \, du \notag \\
		 & \hspace{4em} + \sum_{m^{2} \le X-2H} \sum_{\substack{m^{2} \le f < (m+1)^{2} \\ f \le X-2H \\ f \in \mathscr{Q}_{B}}} H 
		  + \, O \Big(H^{\frac{3}{2}} + X(\log X)^{5} \Big). \label{H}
	\end{align}
	Put the first term to be $\Sigma_{1}$ and the second term to be $\Sigma_{2}$.
	
	First, we consider $\Sigma_{1}$.
	Before evaluating $\Sigma_{1}$, we first introduce following the lemma and the theorem.
	
	\begin{Lem} For $m^{2} \le X-2H$, we have \label{Lem_B}
		$$
		\sum_{\substack{m^{2} \le f < (m+1)^{2} \\ f \le X-2H \\ f \in \mathscr{Q}_{B}}} 1 
		\ll B.
		$$
	\end{Lem}
	
	\begin{proof}
		From the definition of $\mathscr{Q}_{B}$, we have
		\begin{equation*}
			\begin{split}
				 \sum_{\substack{m^{2} \le f < (m+1)^{2} \\ f \le X-2H \\ f \in \mathscr{Q}_{B}}} 1
				&=   \sum_{\substack{m^{2} \le a^{2}b^{3} < (m+1)^{2} \\ b < B}} \mu(b)^{2} \\
				&=   \sum_{b \le  B} \sum_{\frac{m^{2}}{b^{3}} \le a^{2} < \frac{(m+1)^{2}}{b^{3}}} 1 \\
				&=   \sum_{b \le  B} \bigg( \bigg[ \frac{m+1}{b^{\frac{3}{2}}} \bigg] + \bigg[ \frac{m}{b^{\frac{3}{2}}} \bigg] \bigg) \\
				&=   \sum_{b \le  B} \bigg( b^{- \frac{3}{2}} +O(1) \bigg) \\
				&\ll \sum_{b \le  B} 1
				 = B.
			\end{split}
		\end{equation*}
	\end{proof}
	
	\begin{Thm} {\rm (\cite{H})}. \label{Lem_H}
		Let $A$ be any real positive number. For $X^{\frac{1}{6} + \varepsilon} \le H \le X$, we have
		$$
			\int_{X}^{2X} \bigg| \sum_{u < p \le u+H} \log p - H \bigg|^{2} du \ll H^{2} X (\log X)^{-A},
		$$
		where the implicit constant depends on $A$ and $\varepsilon$.
	\end{Thm}
	\begin{proof}
		This follows by a standard argument using Perron's formula and the zero-density estimate of Ingham \cite{I} and Huxley \cite{H}.
	\end{proof}
	
	We now evaluate $\Sigma_{1}$.
	Lemma \ref{Lem_B} gives
	\begin{equation*}
		\begin{split}
			 \Sigma_{1}
			=   & \sum_{m^{2} \le X-2H} \sum_{\substack{m^{2} \le f < (m+1)^{2} \\ f \le X-2H \\ f \in \mathscr{Q}_{B}}} \int_{m}^{m+1} \bigg( \sum_{X-u^{2} < p \le X+H-u^{2}} \log p - H \bigg) \, du \\
			\ll & B \int_{1}^{[ \sqrt{X-2H} ] +1} \bigg| \sum_{X-u^{2} < p \le X+H-u^{2}} \log p - H \bigg| \, du \\
			\ll & B \int_{X - ([\sqrt{X-2H}]+1)^{2}}^{X-1} \bigg| \sum_{t < p \le t+H} \log p -H \bigg| \bigg( \frac{1}{2\sqrt{X-t}} \bigg) \, dt. 
		\end{split}
	\end{equation*}
	
	When $X^{\frac{1}{2} + \varepsilon} \le H$, we have $X - ([\sqrt{X-2H}]+1)^{2} \ge H$ and 
	\begin{equation*}
		\begin{split}
			 \Sigma_{1}
			\ll & B \int_{H}^{X-1} \bigg| \sum_{t < p \le t+H} \log p -H \bigg| \bigg( \frac{1}{2\sqrt{X-t}} \bigg) \, dt.
		\end{split}
	\end{equation*}
	
	From Cauchy-Schwarz inequality,
	\begin{equation*}
		\begin{split}
			 \Sigma_{1}
			\ll & B  \bigg( \bigg( \int_{H}^{X-1} \bigg( \sum_{t < p \le t+H} \log p -H \bigg)^{2} \, dt \bigg) \bigg( \int_{H}^{X-1} \bigg( \frac{1}{\sqrt{X-t}} \bigg)^{2} \, dt \bigg) \bigg)^{\frac{1}{2}} \\
			\ll & B \bigg( \bigg( \int_{H}^{X-1} \bigg( \sum_{t < p \le t+H} \log p -H \bigg)^{2} \, dt \bigg) \bigg)^{\frac{1}{2}} (\log X)^{\frac{1}{2}}. 
		\end{split}
	\end{equation*}
	
	Let $K$ be a positive number such that $2^{K}H \le X-1 < 2^{K+1}H$.
	Then,
	\begin{equation*}
		\begin{split}
			 \Sigma_{1}
			\ll & B \bigg( \bigg( \int_{H}^{2^{K+1}H} \bigg( \sum_{t < p \le t+H} \log p -H \bigg)^{2} \, dt \bigg) \bigg)^{\frac{1}{2}} (\log X)^{\frac{1}{2}} \\
			\ll & B \bigg(\sum_{k=0}^{K} \bigg( \int_{2^{k}H}^{2^{k+1}H} \bigg( \sum_{t < p \le t+H} \log p -H \bigg)^{2} \, dt \bigg) \bigg)^{\frac{1}{2}} (\log X)^{\frac{1}{2}}.
		\end{split}
	\end{equation*}
	
	Since $H \ge X^{\frac{1}{2}} + \varepsilon$ and $2^{k}H \le X$ in the last integrals, we may use Theorem \ref{Lem_H}, to get
	\begin{equation*}
		\begin{split}
			 \Sigma_{1}
			\ll & B \bigg(\sum_{k=0}^{K} \bigg( 2^{k}H^{3} (\log X)^{-A} \bigg) \bigg)^{\frac{1}{2}} (\log X)^{\frac{1}{2}} \\
			\ll & B \bigg( 2^{K}KH^{3} (\log X)^{-A} \bigg)^{\frac{1}{2}} (\log X)^{\frac{1}{2}} \\
			\ll & HX^{\frac{1}{2}}(\log X)^{-A}, 
		\end{split}
	\end{equation*}
	where we replaced the value of $A$ appropriately at the last line.
	
	Next, we consider $\Sigma_{2}$.
	We first rewrite the sum as
	\begin{align}
		\Sigma_{2} 
		= \sum_{m^{2} \le X-2H} \sum_{\substack{m^{2} \le f < (m+1)^{2} \\ f \le X-2H \\ f \in \mathscr{Q}_{B}}} H \,
		= \, \sum_{\substack{f \le X-2H \\ f \in \mathscr{Q}_{B}}} H \,
		= \, \sum_{\substack{f \le X-2H \\ f \in \mathscr{Q}}} H -\sum_{\substack{f \le X-2H \\ f \in \mathscr{Q} \backslash \mathscr{Q}_{B}}} H. \label{Sigma2}
	\end{align}
	
	By using Theorem $\ref{Thm_B-G}$, the first term in (\ref{Sigma2}) is
	\begin{equation*}
		\begin{split}
			\sum_{\substack{f \le X-2H \\ f \in \mathscr{Q}}} H 
			&= \frac{\zeta(\frac{3}{2})}{\zeta(3)}H(X-2H)^{\frac{1}{2}} + O(HX^{\frac{1}{2}}(\log X)^{-1}) \\
			&= \frac{\zeta(\frac{3}{2})}{\zeta(3)}HX^{\frac{1}{2}} + \frac{\zeta(\frac{3}{2})}{\zeta(3)}H \Big( (X-2H)^{\frac{1}{2}} - X^{\frac{1}{2}} \Big) + O(HX^{\frac{1}{2}}(\log X)^{-1}) \\
			&= \frac{\zeta(\frac{3}{2})}{\zeta(3)}HX^{\frac{1}{2}} + O \Big( H^{2}X^{-\frac{1}{2}} + HX^{\frac{1}{2}}(\log X)^{-1} \Big).
		\end{split}
	\end{equation*}
	
	The second term in (\ref{Sigma2}) is
	\begin{equation*}
		\begin{split}
			\sum_{\substack{f \le X-2H \\ f \in \mathscr{Q} \backslash \mathscr{Q}_{B}}} H
			&\ll \sum_{\substack{a^{2}b^{3} \le X \\ B < b}} \mu(b)^{2} H \\
			&\ll \sum_{B < b \le X^{\frac{1}{2}}} \sum_{a^{2} \le \frac{X}{b^{3}}} H \\
			&\ll \sum_{B < b} HX^{\frac{1}{2}} b^{-\frac{3}{2}} \\
			&\ll HX^{\frac{1}{2}} B^{-\frac{1}{2}} \\
			&\ll HX^{\frac{1}{2}}(\log X)^{-2}.
		\end{split}
	\end{equation*}
	
	These results give
	$$
		\Sigma_{2} = \frac{\zeta(\frac{3}{2})}{\zeta(3)}HX^{\frac{1}{2}} + O \Big( H^{2}X^{-\frac{1}{2}} + HX^{\frac{1}{2}}(\log X)^{-1} \Big).
	$$
	
	Then, (\ref{H}) becomes
	\begin{equation*}
		\begin{split}
			  & \sum_{X < N \le X+H} R_{B}(N) \\
			=& \Sigma_{1} + \Sigma_{2} + O \Big(H^{\frac{3}{2}}B + H^{2}X^{- \frac{1}{2}}(\log X)^{5} + X(\log X)^{5} \Big) \\
			=& \frac{\zeta(\frac{3}{2})}{\zeta(3)}HX^{\frac{1}{2}}  \\
			 & + O \Big( HX^{\frac{1}{2}}(\log X)^{-1} +H^{\frac{3}{2}}B + H^{2}X^{- \frac{1}{2}}(\log X)^{5} + X(\log X)^{5} \Big).
		\end{split}
	\end{equation*}

	Taking $X^{\frac{1}{2}+\varepsilon} \le H \le X^{1- \varepsilon}$ gives
		$$
			\sum_{X < N \le X+H} R_{B}(N) = \frac{\zeta(\frac{3}{2})}{\zeta(3)}HX^{\frac{1}{2}} \bigg( 1 + O((\log X)^{-1}) \bigg),
		$$
	which leads to Theorem \ref{main}.

\section*{Acknowledgements}
	The authors would like to thank Prof. Maki Nakasuji for her guidance and helpful advice.


\vspace{1em}

\begin{flushleft}
{\textsc{%
\small
Fumi Ogihara\\[.3em]
\footnotesize
Graduate School of Science and Technology, \\
Sophia University, \\
7-1 Kioicho, Chiyoda-ku, Tokyo 102-8554, Japan
}

\small
\textit{Email address}: \texttt{f-ogihara-8o8@eagle.sophia.ac.jp}
}
\end{flushleft}

\vspace{1em}

\begin{flushleft}
{\textsc{%
\small
Yuta Suzuki\\[.3em]
\footnotesize
Department of Mathematics, \\
Rikkyo University, \\
3-34-1 Nishi-Ikebukuro, Toshima-ku, Tokyo 171-8501, Japan.
}

\small
\textit{Email address}: \texttt{suzuyu@rikkyo.ac.jp}
}
\end{flushleft}

\end{document}